\documentclass[11pt]{article}
\usepackage[utf8]{inputenc}
\usepackage{bbm}
\usepackage{bigints}
\usepackage{amsthm}
\usepackage{amssymb}
\usepackage{color}
\usepackage{tikz}

\newtheorem{theorem}{Theorem}
\newtheorem{proposition}{Proposition}
\newtheorem{corollary}{Corollary}
\newtheorem{lemma}{Lemma}
\newtheorem{hypothesis}{Hypothesis}

\theoremstyle{definition}
\newtheorem{definition}{Definition}

\theoremstyle{remark}
\newtheorem{remark}{Remark}
\def\C{\mathbb{C}}
\def\E{\mathbb{E}}
\def\P{\mathbb{P}}
\def\R{\mathbb{R}}

\def\N{\mathbb{N}}
\def\1{\mathbbm{1}}
\def\d{\partial}
\def\Z{\mathbb{Z}}

\def\cM{{\cal M}}
\def\cB{{\cal B}}

\begin{document}

\title{Quasi-stationarity and quasi-ergodicity for discrete-time Markov chains with absorbing boundaries moving periodically}
\author{William Oçafrain$^{1}$}
\date{\today}


\footnotetext[1]{Institut de Math\'ematiques de Toulouse. CNRS UMR 5219. Universit\'e Paul Sabatier, 118 route de Narbonne, F-31062 Toulouse cedex 09.\\
  E-mail: william.ocafrain@math.univ-toulouse.fr}

\maketitle

\begin{abstract}
    We are interested in quasi-stationarity and quasi-ergodicity when the absorbing boundary is moving. First we show that, in the moving boundary case, the quasi-stationary distribution and the quasi-limiting distribution are not well-defined when the boundary is oscillating periodically. Then we show the existence of a quasi-ergodic distribution for any discrete-time irreducible Markov chain defined on a finite state space in the fixed boundary case. Finally we use this last result to show the quasi-ergodicity in the moving boundary case.   
\end{abstract}

\bigskip

\bigskip

\section{Introduction}

Let $(\Omega, {\cal A}, \P)$ be a probability space and let $X = (X_n)_{n \in \N}$ be a Markov chain with a finite state space $(E, {\cal E})$, ${\cal E}$ being the $\sigma$-algebra containing all the subset of $E$. Let $\P_x$ be the probability measure such that $\P_x(X_0 = x)=1$ and, for any measure $\mu$ on $E$, define $\P_\mu = \int \P_x d\mu(x)$. Denote by $\cM_1(E)$ the set of probability measures defined on $E$. \\
We define, for each time $n \geq 0$, a subset $A_n \subset E$ called \textit{killing subset at time $n$} and we denote by $E_n$ the complement of $A_n$ called \textit{survival subset at time $n$}. We will call $(A_n)_{n \in \N}$ \textit{the moving killing subset} or \textit{the moving killing boundary}. 
We denote by $\tau$ the random variable defined as follows
\begin{equation}
\label{tau-1}
    \tau := \inf\{n \geq 0 : X_n \in A_n\}
\end{equation}
For any subset $B \subset E$, we define $\tau_B$ as follows
\begin{equation*}
    \tau_B := \inf\{n \geq 0 : X_n \in B\}
\end{equation*}
and, to make the notation easier, for any $m \in \N$, we denote by $\tau_m$ the random variable defined by
\begin{equation}
\label{tau-m}
\tau_m := \tau_{A_m} = \inf\{n \geq 0 : X_n \in A_m\}
\end{equation}
 This chapter will deal with quasi-stationary, quasi-limiting and quasi-ergodic distributions that we define as follows. 
\begin{definition}
\label{qsd-1}
$\nu$ is a \textit{quasi-stationary distribution} if for any $n \geq 0$
\begin{equation*}
    \P_\nu(X_n \in \cdot | \tau > n) = \nu(\cdot)
\end{equation*}
\end{definition}

\begin{definition}
\label{qld-1}
$\nu$ is a \textit{quasi-limiting distribution} if there exist some $\mu \in {\cal M}_1(E)$ such that
\begin{equation*}
    \P_\mu(X_n \in \cdot | \tau > n) \underset{n \to \infty}{\longrightarrow} \nu(\cdot)
\end{equation*}
\end{definition}

\begin{definition}
\label{qed-continuous-time}
$\nu$ is a \textit{quasi-ergodic distribution} or a \textit{mean-ratio quasi-stationary distribution} if for any $\mu \in {\cal M}_1(E)$ and any bounded measurable function $f$
\begin{equation*}
    \E_\mu\left(\frac{1}{n} \underset{k=0}{\overset{n-1}{\sum}} f(X_k)| \tau > n\right) \underset{n \to \infty}{\longrightarrow} \int f d\nu
\end{equation*}
\end{definition}
We will also be interested in the existence of a Q-process, which can be interpreted as the process $X$ conditioned never to be absorbed by $(A_n)_{n \in \N}$. \\ \\
In the case where the sequence $(A_n)_{n \in \N}$ does not depend on the time, the existence of these probability measures was established under several assumptions. See for example \cite{CMSM,MV2012} for a general review on the theory of quasi-stationary distributions.  
For modelling purpose, some recent works (see \cite{CCG2016}) introduce some Markov processes absorbed by moving boundaries and the classical theory on quasi-stationary distributions does not allow anymore to describe the asymptotic behavior of the process conditioned not to be absorbed. Our purpose is therefore to check whether these probability measures are still well-defined when $(A_n)_{n \geq 0}$ depends on the time or not. \\\\
In all what follows, we will assume that for any $x \in E_0$,
\begin{equation*}
    \P_x(\tau < \infty) = 1
\end{equation*}
and will also assume that the following hypothesis of irreducibility holds
\begin{equation}
\label{irreducibility-1}
    \forall n \in \N, \forall x,y \in E_n, \exists m \in \N, \P_x(X_{m \land \tau_n}=y) > 0 
\end{equation}

Quasi-stationary distribution will be studied for general moving killing boundaries. However, in a significant part of this chapter we will deal with moving killing boundaries $(A_n)_{n \in \N}$ which are $\gamma$-periodic with $\gamma \geq 2$. \\
In this chapter, we will actually show that there are no quasi-stationary distributions and quasi-limiting distributions in the sense of Definitions \ref{qsd-1} and \ref{qld-1} when the boundaries are moving periodically. However, we will show that the notion of quasi-ergodic distribution and Q-process still makes sense even when the boundary is moving. In particular, we will show the following statement.
\begin{theorem}
Assume that $(A_n)_{n \in \N}$ is $\gamma$-periodic. Then, for some initial law $\mu \in {\cal P}(E)$ and under assumptions which will be spelled out later, there exists a probability measure $\eta$ such that, for any bounded measurable function $f$,
\begin{equation*}
    \E_\mu\left(\frac{1}{n} \underset{k=0}{\overset{n-1}{\sum}} f(X_k) \middle| \tau > n\right) \underset{n \to \infty}{\longrightarrow} \int f d\eta
\end{equation*}
\end{theorem}
The proof is divided to several steps. First we reduce the problem to the study of quasi-stationary distribution in a non moving domain, but for a periodic Markov chain. Then we extend the result proved by Darroch and Senata (see \cite{Darroch1965}) in the aperiodic case to the periodic situation ($\gamma \in \N^*$). \\
This chapter ends with an application of this theorem to random walks absorbed by $2$-periodic moving boundaries.

\section{Quasi-stationary distribution with moving killing subset}

The following proposition shows that the notion of quasi-stationary distribution as defined in Definition \ref{qsd-1} is not relevant when the killing boundary is moving.
\begin{proposition}
\label{noqsd}
Assume there exist $l,m \in \N$ such that $A_l \ne A_m$. Then there is no measure $\nu \in {\cal M}_1(E)$ satisfying the following property:
\begin{equation}
\label{dqs}
    \forall n \in \N, ~~\P_\nu(X_n \in \cdot | \tau > n) = \nu(\cdot)
\end{equation}
\end{proposition}
\begin{proof}
 For any $n \in \N$, denote by $f_n : {\cal M}_1(E) \to {\cal M}_1(E)$ the function defined by
\begin{equation}
\label{f-n}
f_n : \mu \longrightarrow \mathbb{P}_\mu(X_1 \in \cdot | \tau_{n} > 1)
\end{equation}
where $\tau_{n}$ is defined in \eqref{tau-m} and denote by $\mu_n$ the probability measure defined by
\begin{equation}
\mu_n = \P_\mu(X_n \in \cdot | \tau > n)
\label{mu-n}
\end{equation}
By the Markov property, we have for any $n \in \N^*$
\begin{equation*}
\mu_n = \P_{\mu_{n-1}}(X_1 \in \cdot | \tau_{n} > 1) = f_{n}(\mu_{n-1})
\end{equation*}
Thus, by induction, we obtain for any $n \in \N$
\begin{equation*}
\mathbb{P}_\mu(X_{n} \in \cdot | \tau > n) = f_{n} \circ \ldots \circ f_{1}(\mu)
\end{equation*}
We deduce from this equality that
\begin{align*}
\forall n \in \N,~  \P_\nu(X_{n} \in \cdot | \tau > n) = \nu(\cdot)  &\Longleftrightarrow \forall n \in \N,~ f_n(\nu) = \nu \\
&\Longleftrightarrow \forall n \in \N,~ \P_\nu(X_1 \in \cdot | \tau_{n} > 1) = \nu(\cdot)
\end{align*}
In other words, $\nu$ is a quasi-stationary distribution in the moving sense if and only if it is a quasi-stationary distribution in the non-moving sense for all the subsets $A_n$.
In particular, if $\nu$ satisfies \eqref{dqs},
\begin{equation*}
   \nu(\cdot) = \P_\nu(X_1 \in \cdot | \tau_{l} > 1) \text{ and } \nu(\cdot) = \P_\nu(X_1 \in \cdot | \tau_{m} > 1)
\end{equation*}
where $l$ and $m$ have been mentioned in the statement of the proposition. However, since the assumption of irreducibility \eqref{irreducibility-1} holds, the previous statement is impossible since the support of the quasi-stationary distributions are different. 
\end{proof}
\begin{remark}
Proposition \ref{noqsd} can be extended to any general Markov process defined on any space state. In particular, for continuous-time Markov processes defined on a metric space $(E,d)$, we may replace the assumption of irreducibility \eqref{irreducibility-1} by the following assumption
\begin{equation*}
\forall t \in \R_+, \forall x,y \in E_t, \forall \epsilon > 0, \exists t_0 \in \R_+, \P_x(X_{t_0 \land \tau_t} \in B(y,\epsilon)) > 0
\end{equation*}
where $ B(y,\epsilon) := \{z \in E : d(y,z) < \epsilon\}$.
\end{remark}

Notice moreover that we did not need any assumption about the behavior of $(A_n)_{n \in \N}$. In all what follows, we consider that $(A_n)_{n \in \N}$ is $\gamma$-periodic with $\gamma \geq 2$.

\section{Quasi-limiting distribution when the killing subset is moving periodically}

We are now interested in knowing whether the definition of quasi-limiting distribution given in definition \ref{qld-1} is meaningful when the killing subset is moving or not. 
In the usual case, it is well known (see \cite{MV2012} p. 345) that quasi-stationary distribution and quasi-limiting distribution are equivalent notions. This implies that the non-existence of a quasi-stationary distribution implies the non-existence of any quasi-limiting distribution. However, this equivalence does not hold anymore in the moving case. Consider for example  $(A_n)_{n \geq 0}$ such that there exists $n_0$ such that for any $n \geq n_0$, $A_n = A_{n_0}$
and assume that there exists a quasi-stationary distribution $\nu_{n_0}$ (in the non-moving sense) such that for any probability measure $\mu$ on $E_{n_0}$,
\begin{equation*}
    \P_{\mu}(X_n \in \cdot | \tau_{n_0} > n) \underset{n \to \infty}{\longrightarrow} \nu_{n_0}
\end{equation*}
Thus, by the Markov property, for any $\mu \in {\cal M}_1(E)$ such that $\P_\mu(\tau > k) > 0$ for all $k \in \N$ and any $n \geq 0$,
\begin{equation*}
    \P_\mu(X_{n+n_0} \in \cdot | \tau > n+n_0) = \P_{\mu_{n_0}}(X_n \in \cdot | \tau_{A_{n_0}} > n) \underset{n \to \infty}{\longrightarrow} \nu_{n_0} 
\end{equation*}
where $\mu_n$ is defined in \eqref{mu-n} for any $n \in \N$.\\ From now on, we will assume that $(A_n)_{n \in \N}$ is periodic and will denote by $\gamma$ its period. We will show that quasi-limiting distribution is not well defined when the killing subset is moving periodically.
\begin{proposition}
\label{noqld}
Assume $(A_n)_{n \in \N}$ is $\gamma$-periodic and there exist $0 \leq l,m \leq \gamma -1$ such that $A_l \ne A_m$. \\
Then there is no $\nu \in {\cal M}_1(E)$ satisfying the following property:
\begin{equation*}
    \exists \mu \in {\cal M}_1(E), ~~\P_\mu(X_n \in \cdot | \tau > n) \underset{n \to \infty}{\longrightarrow} \nu(\cdot) 
\end{equation*}
\end{proposition}

\begin{proof}
Consider again the functions $f_m$ defined in \eqref{f-n}:
\begin{equation*}
    f_m : \mu \longrightarrow \P_\mu(X_1 \in \cdot | \tau_m > 1)
\end{equation*}
Then using the periodicity of $(A_n)_{n \in \N}$ and by the Markov property, for any $k \in \{1, \ldots, \gamma\}$, $m \in \mathbb{N^*}$ and $\mu \in {\cal M}_1(E)$
\begin{equation}
\mathbb{P}_\mu(X_{k + m \gamma} \in \cdot | \tau > k+m \gamma) = g_k \circ f^m(\mu)
\end{equation}
with 
\begin{align*}
&g_k = f_{k} \circ \ldots \circ f_{1}\\
&f = f_{\gamma} \circ \ldots \circ f_{1}
\end{align*}
 Assume that there exists $\mu \in {\cal M}_1(E)$ such that the sequence $(\mathbb{P}_\mu(X_{m} \in \cdot | \tau > m))_{m \in \mathbb{N}}$ converges to its limit $\nu$. Then
\begin{align*}
\nu &= \underset{m \to \infty}{\lim} \mathbb{P}_\mu(X_{m\gamma} \in \cdot | \tau > m\gamma) \\
&= \underset{m \to \infty}{\lim} f^m(\mu)
\end{align*}
So for any $k \in \{1, \ldots, \gamma\}$
\begin{equation*}
\nu = g_k(\nu) = \mathbb{P}_\nu(X_{k} \in \cdot | \tau > k)
\end{equation*}
In other words, for any $k \in \{1, \ldots, \gamma\}$,
\begin{equation*}
\nu = f_{k}(\nu)
\end{equation*}
We thus may conclude in the proof of proposition \ref{noqsd}.
\end{proof}

The previous statement implies therefore that the quasi-limiting distribution as defined in Definition \ref{qld-1} is not well-defined when the moving killing subset is periodic. However, according to the proof of the previous proposition, it seems that the sequence of these conditioned probabilities could have some limit points. \\\\
The following proposition allows us to pass from a moving problem to a non-moving problem. The existence of limit points will be therefore a consequence of the existence of classical quasi-stationary distributions for the transformed Markov chain. 
\begin{proposition}
For any $0 \leq m \leq \gamma-1$ and $\mu \in {\cal M}_1(E)$, there is a Markov chain $(X^{(m)}_n)_{n \in \mathbb{N}}$ such that 
\begin{equation}
\mathbb{P}_\mu((X_m, X_{m+\gamma}, \ldots, X_{m + n\gamma}) \in \cdot | \tau > m + n\gamma) = \mathbb{P}_{\mu_m}((X^{(m)}_0, X^{(m)}_1, \ldots, X^{(m)}_n) \in \cdot | \tau^{(m)}_{m} > n)
\label{egalite_lemme}
\end{equation}
where $\mu_m$ is defined in \eqref{mu-n} and $\tau^{(m)}_{m} = \inf\{n \in \N : X^{(m)}_n \in A_m\}$. 
\end{proposition}
\begin{proof}
According to the Markov property, it is enough to show that for any $\gamma$-periodic sequence of subsets $B = (B_n)_{n \in \N}$ and any measure $\mu \in {\cal M}_1(E)$, there exists a Markov chain $(Z_n)_{n \in \mathbb{N}}$ such that
\begin{equation}
\mathbb{P}_\mu((X_\gamma, \ldots, X_{n\gamma}) \in \cdot | \tau(B) >  n\gamma) = \mathbb{P}_\mu((Z_1, \ldots, Z_n) \in \cdot | \tilde{\tau}_{B_0} > n)
\label{egalite-lemme-simple}
\end{equation}
where $\tau(B) = \inf\{m \geq 0 : X_m \in B_m\}$ and $\tilde{\tau}_{B_0} = \inf\{n \in \N : Z_n \in B_0\}$. Denote $F_0$ the complement of $B_0$. 
For any $x \in F_0$ define $p(x,\cdot)$ by
\begin{align*}
    &p(x,A) = \mathbb{P}_x(X_\gamma \in A, \tau_{B} > \gamma), ~~\forall A \subset F_0 \\
    & p(x,B_0) = 1-p(x,F_0)
\end{align*} and we denote by $(Z_n)_{n \in \mathbb{N}}$ the Markov chain for which the transition kernel is $p$. We will show by induction that, for any $\phi_1, \ldots, \phi_n$ bounded measurable functions, 
\begin{equation*}
    \mathbb{E}_\mu(\phi_1(X_\gamma) \ldots \phi_n(X_{n\gamma}) \mathbbm{1}_{\tau(B) > n\gamma}) = \mathbb{E}_\mu(\phi_1(Z_1) \ldots \phi_n(Z_{n}) \mathbbm{1}_{\tilde{\tau}_{B_0} > n})
\end{equation*}
By definition of $(Z_n)_{n \in \N}$, for any probability measure $\mu$ and any bounded measurable function $\phi$,
$$\E_\mu(\phi(Z_1)\1_{\tilde{\tau}_{B_0} > 1}) = \E_\mu(\phi(X_\gamma) \1_{\tau(B) > \gamma})$$
which entails the base case.
Now assume that the equality for $n-1$ is satisfied. Let $\phi_1, \ldots, \phi_n$ be some bounded measurable functions. Then
\begin{align*}
    \mathbb{E}_\mu(\phi_1(X_\gamma) \ldots \phi_n(X_{n\gamma}) \mathbbm{1}_{\tau(B) > n\gamma}) &= \mathbb{E}_\mu(\phi_1(X_\gamma) \mathbbm{1}_{\tau(B) > \gamma} \mathbb{E}_{X_\gamma}(\phi_2(X_\gamma) \ldots \phi_n(X_{(n-1)\gamma}) \mathbbm{1}_{\tau(B) > (n-1)\gamma})) \\
    &=\mathbb{E}_\mu(\phi_1(Z_1) \mathbbm{1}_{\tilde{\tau}_{B_0} > 1} \mathbb{E}_{Z_1}(f_2(Z_1) \ldots \phi_n(Z_{n-1}) \mathbbm{1}_{\tilde{\tau}_{B_0} > (n-1)})) \\
    &= \mathbb{E}_\mu(\phi_1(Z_1) \ldots \phi_n(Z_{n}) \mathbbm{1}_{\tilde{\tau}_{B_0} > n})
\end{align*}
This concludes the proof.
\end{proof}

\section{Existence of quasi-ergodic distribution with periodic moving killing subsets}
In this section, our aim is to show the existence of a quasi-ergodic distribution as defined in Definition \ref{qed-continuous-time} when the boundary is moving periodically. This section will be split into three parts : 
\begin{enumerate}
\item We will first study quasi-ergodicity in the non-moving case (when $A_n = A_0, \forall n \in \N$) for irreducible Markov chains.
\item Then we will use the results obtained in the first part to deduce quasi-ergodicity for general Markov chains (irreducible or not), but still considering non-moving boundaries
\item Finally we will show the existence of quasi-ergodic distribution when $(A_n)_{n \in \N}$ is moving periodically.  
 \end{enumerate}
\subsection{Quasi-ergodic distribution in the classical non-moving sense for irreducible periodic Markov chains}
\label{1-1}
In this subsection we will study the quasi-ergodicity of one irreducible Markov chain $Y = (Y_n)_{n \in \N}$ in the classical non-moving sense, that is when the killing edge does not move. Without loss of generality, assume $Y$ is defined in the state space $E_0 =\{0, \ldots, K\}$ and that the cemetery is $\{0\}$. In this subsection and the following, $\tau$ will be defined as \eqref{tau-1} but refering to $Y$, that is
\begin{equation*}
\tau = \inf\{n \geq 0 : Y_n = 0\}
\end{equation*}
 Denote by $P$ the transition matrix of $Y$. Since $0$ is an absorbing state for $Y$, $P$ has the following form
\begin{equation*}
    P = \begin{pmatrix}
    1 & 0 \\
    v & Q
    \end{pmatrix}
\end{equation*}
where $Q$ is \textit{the sub-transition matrix}. We will assume that $Q$ is irreducible (i.e. $\forall x,y \in E_0, \exists n \in \N, Q^n(x,y) > 0$). As a result we can define $T_x$ the period of $x$ as
\begin{equation*}
    T_x := gcd\{n \in \N :  \P_x(Y_n = x, \tau > n) > 0\}
\end{equation*}
where $gcd$ refers to the greatest common divisor. By irreducibility of $Q$, all the $x$ have the same period and we denote by $T$ this common period.\\
The existence of quasi-ergodic distributions has already been proved by Darroch and Seneta in \cite{Darroch1965} when $T=1$. However we will see that this result is not enough for our purpose and we need to extend it to periodic Markov chains.\\
Due to the periodicity of $Q$, there exist $(C_i)_{0 \leq i \leq T-1}$ a partition of $E_0$ such that if the support of the initial distribution $\mu$ is included in $C_0$, then for any $n \in \N$ and $0 \leq k \leq T-1$,
\begin{equation*}
\P_\mu(Y_{k+nT} \in C_k, \tau > k+nT) = 1
\end{equation*} 
Without loss of generality, we construct $(C_i)_{0 \leq i \leq T-1}$ such that $1 \in C_0$. Formally $(C_i)_{0 \leq i \leq T-1}$ are defined by
\begin{align}
\label{cluster0}
    &C_0 := \{y \in E_0 : \exists n \in \N^*, \P_1(Y_{nT} = y,\tau > nT)> 0\} \\
\label{cluster}
    &\forall 1 \leq i \leq T-1,~~C_i := \{y \in E_0 : \exists x \in C_{i-1}, \P_x(Y_1 = y)>0\} 
\end{align}
For each $j \in \{0, \ldots, T-1\}$ and any $v \in \C^K$, we will denote by $v^{(j)}$ the sub-vector of $v$ restricted to $C_j$. \\
It is well known by the Perron-Frobenius theorem that the spectral radius $$\rho := \max\{|\lambda| : \lambda \in Sp(Q)\}$$ is a simple eigenvalue of $Q$ and that one can find a left-eigenvector $\nu = (\nu(j))_{1 \leq j \leq K}$ and a right-eigenvector $\xi = (\xi(j))_{1 \leq j \leq K}$ for $\rho$ (i.e. $\nu Q = \rho \nu$ and $Q\xi = \rho \xi$) such that $\nu(j) > 0$ and $\xi(j) > 0$ for all $j \in \{1, \ldots, K\}$. We may choose $\nu$ and $\xi$ such that
\begin{equation*}
 <\nu,\1> = <\nu,\xi> = 1
\end{equation*}
 where $<\cdot,\cdot>$ is the usual Hermitian product on $\C^K$. Moreover, since $Q$ is $T$-periodic, 
\begin{equation*}
    \{\lambda_k := \rho e^{\frac{2ik\pi}{T}} : 0 \leq k \leq T-1\} \subset Sp(Q)
\end{equation*}
and each $\lambda_k$ is simple. For each $\lambda_k$ we can obtain a left eigenvector $v_k$ and a right-eigenvector $w_k$ from $\nu$ and $\xi$ with the following transformation
\begin{equation}
\label{eigenvectors}
\forall j \in \{0, \ldots, T-1\}, v_k^{(j)} = e^{-i\frac{2\pi j k}{T}}\nu^{(j)} \text{ and } w_k^{(j)} = e^{i\frac{2\pi j k}{T}}\xi^{(j)}
\end{equation}
See Theorem 1.7 in [\cite{Seneta1980},p.23-24] for more details. \\
The vectors $(v_i)_{0 \leq i \leq T-1}$ are linearly independent. We can complete this family into a basis ${\cal V} = (v_i)_{0 \leq i \leq K-1}$ such that $v_i \in Span^\perp(v_0, \ldots, v_{T-1})$ for all $T \leq i \leq K-1$ where
\begin{equation*}
 Span^\perp(v_0, \ldots, v_{T-1}) = \{v \in \C^K : <v,v_i> = 0, ~\forall i \in \{0, \ldots, T-1\}\}
\end{equation*}
 Let us denote by $R$ the matrix representing the change of basis from the canonical basis to $\cal V$. Then we have the following decomposition
\begin{equation*}
    Q = R \begin{pmatrix}
    \begin{matrix}
    \lambda_0 & & \\
    & \ddots & \\
    & & \lambda_{T-1}
    \end{matrix} & 0\\
    0 & Q_0
    \end{pmatrix}
    R^{-1}
\end{equation*}
where $Q_0$ is a $(K-T) \times (K-T)$ matrix. We define the matrix $Q'$ by
\begin{equation*}
    Q' = R \begin{pmatrix}
    0 & 0 \\
    0 & Q_0
    \end{pmatrix} R^{-1}
\end{equation*}

\begin{proposition}
\label{quasi-ergodicity}
Let $f : \{1, \ldots, K\} \to \R$ be a bounded measurable function. Then for any $x \in \{1, \ldots, K\}$ and $n \in \N^*$,  
\begin{equation*}
    \E_x\left(\frac{1}{n} \underset{k=0}{\overset{n-1}{\sum}} f(Y_k) \1_{\tau > n} \right) = \rho^n \varphi(f) \underset{l=0}{\overset{T-1}{\sum}} e^{-\frac{2inl\pi}{T}}  <w_l,\delta_x> <v_l,\1> + o(\rho^n) 
\end{equation*}
where $$\varphi(f) = \underset{i=1}{\overset{K}{\sum}} f(i)\nu(i) \xi(i)$$.
\end{proposition}
\begin{proof}
Let $f:\{1, \ldots, K\} \to \R$ be a bounded measurable function. In this proof we will consider probability measures on $\{1, \ldots, K\}$ and functions defined on $\{1, \ldots, K\}$ as $K$-vectors. Thus for any $x \in \{1, \ldots, K\}$ we can say
\begin{equation}
\label{label}
    \E_{x}\left(f(Y_n)\1_{\tau > n}\right) = <\delta_x Q^n, f>
\end{equation}
where $\delta_x$ is the Dirac measure on $x$.
For any $x \in \{1, \ldots, K\}$, define $(\alpha_k(x))_{0 \leq k \leq T-1}$ the unique family of $\C^K$ such that there is $\mu_x \in Span^\perp(v_0, \ldots, v_{T-1})$ such that 
\begin{equation*}
\delta_x = \underset{k=0}{\overset{T-1}{\sum}} \alpha_k(x) v_k + \mu_x
\end{equation*}
We will use the following lemma whose proof is postponed after the proof of the theorem.
\begin{lemma}
\label{lemma-eigenvector}
For any $0 \leq k \leq T-1$,
\begin{equation*}
    (\alpha_k(x))_{x \in E_0} = w_k
\end{equation*}
where $w_k$ is defined in \eqref{eigenvectors} 
\end{lemma}
Thus we can write
\begin{equation}
\label{decomposition}
\delta_x = \underset{k=0}{\overset{T-1}{\sum}} w_k(x) v_k + \mu_x
\end{equation}
So, using \eqref{label} and \eqref{decomposition}, for any $n \in \N$
\begin{align*}
\mathbb{E}_{x}\left(f(X_n)\mathbbm{1}_{\tau > n}\right)
&= <\underset{k=0}{\overset{T-1}{\sum}} w_k(x) v_k Q^n, f> + <\mu_x Q^n,f> \\
&= \underset{k=0}{\overset{T-1}{\sum}} \overline{\lambda_k^n w_k(x)} <v_k, f> + <\mu_x(Q')^n,f> 
\end{align*} 
Now, using the Markov property, for any $k \leq n$, 
\begin{align}
\mathbb{E}_\mu(f(Y_k)\mathbbm{1}_{\tau > n}) &= \mathbb{E}_\mu(\mathbbm{1}_{\tau > k} f(Y_k) \mathbb{P}_{Y_k}(\tau > n-k)) \\
&=\mathbb{E}_\mu(\mathbbm{1}_{\tau > k}g_{k,n}(f)(Y_k))
\label{Markov_property}
\end{align}
where, for any $y \in E_0$,  $$g_{k,n}(f)(y) = f(y)\P_y(\tau > n-k)$$ Then,
\begin{align*}
    g_{k,n}(f)(y)
    &= f(y) <\delta_y Q^{n-k}, \mathbbm{1}> \\
    &= \underset{m=0}{\overset{T-1}{\sum}} \overline{\lambda_m^{n-k}} f(y) \overline{w_m(y)}<v_m,\mathbbm{1}> + f(y)<\mu_y (Q')^{n-k},\1>
\end{align*}
Define, for any $k \in \{0, \ldots, T-1\}$ and $n \in \N$,
\begin{align*}
    &g_k(f):y \to f(y)\overline{w_l(y)} \\
    &w_n(f) : y \to f(y) <\mu_y(Q')^n,\1>
\end{align*}Then, using \eqref{Markov_property}, for any $k \leq n$,
\begin{align*}
    \mathbb{E}_x(f(Y_k)\mathbbm{1}_{\tau > n}) &=<\delta_x Q^k, g_{k,n}(f)>\\ &= \underset{l=0}{\overset{T-1}{\sum}} \overline{\lambda_l^k w_l(x)} <v_l, g_{k,n}(f)> + <\mu_x(Q')^k,g_{k,n}(f)> \\
    &= A_{k,n}+B_{k,n}+C_{k,n}+D_{k,n}
\end{align*}
where
\begin{align*}
    &A_{k,n} = \underset{l=0}{\overset{T-1}{\sum}}\underset{m=0}{\overset{T-1}{\sum}} \overline{\lambda_l^k \lambda_m^{n-k}w_l(x)}<v_l,g_m(f)><v_m,\1> \\
    &B_{k,n} = \underset{l=0}{\overset{T-1}{\sum}} \overline{\lambda_l^k w_l(x)} <v_l, w_{n-k}(f)> \\
    &C_{k,n} = \underset{m=0}{\overset{T-1}{\sum}}\overline{\lambda_m^{n-k}} <v_m,\mathbbm{1}> <\mu_x(Q')^k,g_m(f)> \\
    &D_{k,n} = <\mu_x(Q')^k,w_{n-k}(f)>
\end{align*}
Hence for any $n \in \N^*$
\begin{equation}
\label{sum}
    \underset{k=0}{\overset{n-1}{\sum}} \mathbb{E}_x(f(Y_k)\mathbbm{1}_{\tau > n}) = \underset{k=0}{\overset{n-1}{\sum}} A_{k,n} +\underset{k=0}{\overset{n-1}{\sum}} B_{k,n} +\underset{k=0}{\overset{n-1}{\sum}} C_{k,n} +\underset{k=0}{\overset{n-1}{\sum}} D_{k,n} 
\end{equation}
\textit{i) Study of $\underset{k=0}{\overset{n-1}{\sum}} A_{k,n}$}

For any $n \in \N^*$,
\begin{align*}
\underset{k=0}{\overset{n-1}{\sum}} A_{k,n} &=  \underset{l=0}{\overset{T-1}{\sum}}\underset{m=0}{\overset{T-1}{\sum}} \left(\underset{k=0}{\overset{n-1}{\sum}}\overline{\lambda_l^k \lambda_m^{n-k}}\right)\overline{w_l(x)}<v_l,g_m(f)><v_m,\1> \\
&= \underset{l=0}{\overset{T-1}{\sum}} n \overline{\lambda_l^n w_l(x)} <v_l,g_l(f)><v_l,\1> + \underset{l \ne m}{\sum} \overline{\lambda_m} \left(\frac{\overline{\lambda_l^{n}} - \overline{\lambda_m^{n}}}{\overline{\lambda_l}-\overline{\lambda_m}}\right) \overline{w_l(x)}<v_l,g_m(f)><v_m,\1> 
\end{align*}
On the one side,
$$\underset{l=0}{\overset{T-1}{\sum}} n \overline{\lambda_l^n w_l(x)} <v_l,g_l(f)><v_l,\1> = n\rho^n \underset{l=0}{\overset{T-1}{\sum}} e^{-\frac{2inl\pi}{T}}  \overline{w_l(x)} <v_l,g_l(f)><v_l,\1>$$
On the other side, for any $0 \leq l \ne m \leq T-1$,
\begin{align*}
\overline{\lambda_m} \left(\frac{\overline{\lambda_l^{n}} - \overline{\lambda_m^{n}}}{\overline{\lambda_l}-\overline{\lambda_m}}\right) &= \rho  e^{-\frac{2im\pi}{T} }\left(\frac{\rho^n  e^{-\frac{2inl\pi}{T} } -\rho^n  e^{-\frac{2inm\pi}{T} }}{\rho  e^{-\frac{2il\pi}{T} }-\rho  e^{-\frac{2im\pi}{T} }}\right) \\
&= \rho^n  e^{-\frac{2im\pi}{T} }\left(\frac{  e^{-\frac{2inl\pi}{T} } -  e^{-\frac{2inm\pi}{T} }}{  e^{-\frac{2il\pi}{T} }-  e^{-\frac{2im\pi}{T} }}\right)
\end{align*}
$\left(e^{-\frac{2im\pi}{T} }\left(\frac{  e^{-\frac{2inl\pi}{T} } -  e^{-\frac{2inm\pi}{T} }}{  e^{-\frac{2il\pi}{T} }-  e^{-\frac{2im\pi}{T} }}\right)\right)_{n \in \N}$ is bounded, hence 
$$\frac{1}{n} \times e^{-\frac{2im\pi}{T} }\left(\frac{  e^{-\frac{2inl\pi}{T} } -  e^{-\frac{2inm\pi}{T} }}{  e^{-\frac{2il\pi}{T} }-  e^{-\frac{2im\pi}{T} }}\right) \underset{n \to \infty}{\longrightarrow} 0$$
We deduce that, for any $0 \leq l \ne m \leq T-1$,
$$\rho^n  e^{-\frac{2im\pi}{T} }\left(\frac{  e^{-\frac{2inl\pi}{T} } -  e^{-\frac{2inm\pi}{T} }}{  e^{-\frac{2il\pi}{T} }-  e^{-\frac{2im\pi}{T} }}\right) = o(n\rho^n)$$
and therefore
$$\underset{l \ne m}{\sum} \overline{\lambda_m} \left(\frac{\overline{\lambda_l^{n}} - \overline{\lambda_m^{n}}}{\overline{\lambda_l}-\overline{\lambda_m}}\right) \overline{w_l(x)}<v_l,g_m(f)><v_m,\1> = o(n\rho^n)$$
since this is a finite sum. Hence
$$\underset{k=0}{\overset{n-1}{\sum}} A_{k,n} =  n\rho^n \underset{l=0}{\overset{T-1}{\sum}} e^{-\frac{2inl\pi}{T}}  \overline{w_l(x)} <v_l,g_l(f)><v_l,\1> + o(n\rho^n)$$
\textit{ii) Study of $\underset{k=0}{\overset{n-1}{\sum}} B_{k,n}$}\\
For any $y \in E$, $n \in \N$ and $0 \leq l \leq T-1$
\begin{align*}
    \underset{k=1}{\overset{n-1}{\sum}} \overline{\lambda_l^k} w_{n-k}(f)(y) &= f(y) <\mu_y \left(\underset{k=0}{\overset{n-1}{\sum}}\lambda_l^k (Q')^{n-k}\right),\1> \\
    &= f(y) <\mu_yQ'(\lambda_l I_{K} - Q')^{-1}(\lambda_l^{n} I_{K}-(Q')^{n}),\1> 
\end{align*}
where $I_{K}$ is the $K \times K$-identity matrix. For any $0 \leq l \leq T-1$ and $n \in \N$,
$$\lambda_l^n I_k - (Q')^n = \rho^n (e^\frac{2i\pi nl}{T} I_k - \rho^{-n}(Q')^n)$$
and $(e^\frac{2i\pi nl}{T} I_k - \rho^{-n}(Q')^n)_{n \in \N}$ is bounded since the spectral radius of $Q'$ is smaller than $\rho$. Hence 
$$\frac{1}{n} \left(e^\frac{2i\pi nl}{T} I_k - \rho^{-n}(Q')^n\right) \underset{n \to \infty}{\longrightarrow} 0$$
where $0$ is understood as the zero matrix, and we deduce that 
$$<\mu_yQ'(\lambda_l I_{K} - Q')^{-1}(\lambda_l^{n} I_{K}-(Q')^{n}),\1> = o(n\rho^n)$$
As a result, for any $n \in \N$,
$$\underset{k=1}{\overset{n-1}{\sum}} \overline{\lambda_l^k} w_{n-k}(f)(y) = o(n\rho^n)$$
Hence for any $n \in \N$
\begin{equation*}
    \underset{k=0}{\overset{n-1}{\sum}} B_{k,n} = \underset{l=0}{\overset{T-1}{\sum}} \overline{w_l(x)} <v_l,\underset{k=0}{\overset{n-1}{\sum}} \overline{\lambda_l^k} w_{n-k}(f)> = o(n\rho^n) 
\end{equation*}
\textit{iii) Study of $\underset{k=0}{\overset{n-1}{\sum}} C_{k,n}$}\\
In the same way as $\underset{k=0}{\overset{n-1}{\sum}} B_{k,n}$,
\begin{align*}
\underset{k=0}{\overset{n-1}{\sum}} C_{k,n} &= \sum_{k=0}^{n-1} \sum_{m=0}^{T-1} \overline{\lambda_m^{n-k}} <v_m, \1> <\mu_x (Q')^k, g_m(f)> \\
&= \sum_{m=0}^{T-1} <v_m,\1> <\mu_x \left( \sum_{k=0}^{n-1} \lambda_m^{n-k} (Q')^k\right), g_m(f)>
\end{align*}
For any $0 \leq m \leq T-1$ and $n \geq 1$,
$$\sum_{k=0}^{n-1} \lambda_m^{n-k} (Q')^k = \lambda_m \times (\lambda_m I_{K} - Q')^{-1} (\lambda_m^{n} I_{K}-(Q')^{n})$$
We already showed that for any $0 \leq m \leq T-1$ and $n \geq 1$
 $$(\lambda_m I_{K} - Q')^{-1} (\lambda_m^{n} I_{K}-(Q')^{n}) = o(n \rho^n)$$
Finally,
\begin{equation*}
    \underset{k=0}{\overset{n-1}{\sum}} C_{k,n} = o(n\rho^n)
\end{equation*}
\textit{iv) Study of $\underset{k=0}{\overset{n-1}{\sum}} D_{k,n}$}\\
Finally, let us denote by $(q')_{i,j}^{(n)}$, for $i,j \in \{1, \ldots, K-T\}$ and $n \in \N$, the coefficient of $(Q')^n$ located at the $i$th row and the $j$th column. Then for any $n \in \N$
\begin{equation}
\label{sumD}
    \underset{k=0}{\overset{n-1}{\sum}} D_{k,n} = \underset{i,j,l,m}{\sum} \mu_x(j) f(i) \mu_i(m) \left(\underset{k=0}{\overset{n-1}{\sum}}(q')_{m,l}^{(n-k)}(q')_{i,j}^{(k)}\right)  
\end{equation}
Let $i,j,l,m \in \{1, \ldots, K\}$. By definition of the matrix $Q'$, the spectral radius of $Q'$ is strictly smaller than $\rho$. We deduce from this 
\begin{equation}
\label{equivalence}
    (q')^{(n)}_{i,j} = o(\rho^n), ~~(q')^{(n)}_{m,l} = o(\rho^n)
\end{equation}
In particular there is a positive number $C$ such that for any $n \in \N$ and $m,l \in \{1, \ldots, K\}$,
\begin{equation*}
    (q')^{(n-k)}_{m,l} \leq C \rho^{n-k}
\end{equation*}
Hence,
\begin{align}
   \underset{k=0}{\overset{n-1}{\sum}}(q')_{m,l}^{(n-k)}(q')_{i,j}^{(k)} &\leq C\underset{k=0}{\overset{n-1}{\sum}}\rho^{n-k}(q')_{i,j}^{(k)} \\
   &= C n \rho^n \left(\frac{1}{n}\underset{k=1}{\overset{n-1}{\sum}}\rho^{-k}(q')_{i,j}^{(k)}\right)
   \label{majoration}
\end{align}
However, by \eqref{equivalence}, $\rho^{-n}q_{i,j}^{(n)} \rightarrow 0$ when $n$ tends to infinity and using Cesaro's lemma,
\begin{equation*}
    \frac{1}{n}\underset{k=0}{\overset{n-1}{\sum}}\rho^{-k}(q')_{i,j}^{(k)} \underset{n \to \infty}{\longrightarrow} 0
\end{equation*}
Thus using \eqref{sumD} and \eqref{majoration}, we deduce that
\begin{equation*}
    \underset{k=0}{\overset{n-1}{\sum}} D_{k,n} = o(n\rho^n)
\end{equation*}
Hence, gathering all these results and using \eqref{sum},
\begin{equation}
\label{asymptotic-sum}
    \underset{k=0}{\overset{n-1}{\sum}} \mathbb{E}_x(f(Y_k)\mathbbm{1}_{\tau > n}) = n\rho^n \underset{l=0}{\overset{T-1}{\sum}} e^{-\frac{2inl\pi}{T}}  \overline{w_l(x)} <v_l,g_l(f)><v_l,\1> + o(n\rho^n)
\end{equation}
However we have for any $l \in \{0, \ldots, T-1\}$
\begin{align*}
    <v_l, g_l(f)> &= \underset{j=1}{\overset{K}{\sum}} f(j) \overline{v_l(j) w_l(j)} \\
    &= \underset{j=0}{\overset{T-1}{\sum}} \underset{x \in C_j}{\sum} f(x) \overline{v_l(x) w_l(x)} \\
    &= \underset{j=0}{\overset{T-1}{\sum}} \underset{x \in C_j}{\sum} f(x) \overline{e^{-i\frac{2\pi j l}{T}} \nu(x) e^{i\frac{2\pi j l}{T}} \xi(x)} \\
    &= <v_0, g_0(f)>
\end{align*}
As a result,
\begin{equation*}
    \E_x\left(\underset{k=0}{\overset{n-1}{\sum}} f(Y_k) \1_{\tau > n} \right) = n\rho^n <v_0,g_0(f)> \underset{l=0}{\overset{T-1}{\sum}} e^{-\frac{2inl\pi}{T}}  \overline{w_l(x)} <v_l,\1> + o(n\rho^n)
\end{equation*}
\end{proof}
Now we prove Lemma \ref{lemma-eigenvector} quoted in the previous proof.
\begin{proof}[Proof of Lemma \ref{lemma-eigenvector}]
Let us start by proving that $\alpha_l$ is a right-eigenvector associated to $\lambda_l$. Since $Q$ is a real matrix, it is equivalent to show that $\overline{\alpha_l}$ is a right-eigenvector associated to $\overline{\lambda_l}$. \\
First remind that $\alpha_l$ is defined by the relations
\begin{equation*}
    \delta_k = \underset{l=0}{\overset{T-1}{\sum}} \alpha_l(k) v_l + \delta_k'
\end{equation*}
for any $k \in E_0$ and with $\delta_k' \in Span^\perp(v_0, \ldots, v_{T-1})$. This implies for any $k$ 
\begin{equation*}
    <\delta_k, v_m> = \underset{l=0}{\overset{T-1}{\sum}} \overline{\alpha_l(k)} <v_l,v_m>
\end{equation*}
or, in other words,
\begin{equation*}
    \begin{pmatrix}
    <\delta_k, v_0> \\
    \vdots \\
    <\delta_k, v_{T-1}>
    \end{pmatrix} = \begin{pmatrix}
    <v_0,v_0> & \ldots & <v_{T-1},v_{0}> \\
    \vdots & \ddots & \vdots \\
    <v_{0}, v_{T-1}> & \ldots & <v_{T-1}, v_{T-1}> 
    \end{pmatrix}
    \begin{pmatrix}
    \overline{\alpha_0(k)} \\
    \vdots \\
    \overline{\alpha_{T-1}(k)}
    \end{pmatrix}
\end{equation*}
Denote by $A$ the matrix
\begin{equation*}
    A = \begin{pmatrix}
    <v_0,v_0> & \ldots & <v_{T-1},v_{0}> \\
    \vdots & \ddots & \vdots \\
    <v_{0}, v_{T-1}> & \ldots & <v_{T-1}, v_{T-1}> 
    \end{pmatrix}
\end{equation*}
$A$ is simply the Gram's matrix of the basis $(v_i)_{0 \leq i \leq T-1}$. Thus the determinant $\det(A)$ is positive and for any $x \in E_0$
\begin{equation*}
    \overline{\alpha_l(x)} = \frac{1}{\det(A)}\begin{vmatrix}
    <v_0,v_0> & \ldots  & <\delta_x, v_0>  & \ldots & <v_{T-1},v_0> \\
    \vdots & \ddots & \vdots  & \ddots & \vdots \\
    <v_0,v_{T-1}> & \ldots  & <\delta_x, v_{T-1}> &  \ldots & <v_{T-1},v_{T-1}>
    \end{vmatrix}
\end{equation*}
where the column $(<\delta_x, v_0>, \ldots, <\delta_x,v_{T-1}>)^T$ is the $l$-th columns of the matrix.
We want to show now that $\alpha_l$ is a right-eigenvector associated to $\lambda_l$, that is
\begin{equation}
\label{eigenvalue}
    \forall v \in \C^K, <v,Q\overline{\alpha_l}> = \overline{\lambda_l} <v,\overline{\alpha_l}>
\end{equation}
In fact it is enough to show \eqref{eigenvalue} when $v$ is one of left-eigenvectors and when $v \in Span^\perp(v_0, \ldots, v_{T-1})$. 
In the case where $v = v_k$ for $k \in \{0, \ldots, T-1\}$
\begin{align*}
    <v_k, \overline{\alpha_l}> &= \underset{j=1}{\overset{K}{\sum}} \overline{v_k(j)} \frac{1}{\det(A)} \begin{vmatrix}
    <v_0, v_0> &  \ldots & <\delta_j, v_0> & \ldots  &  <v_{T-1},v_0> \\
    \vdots &  \ddots  & \vdots & \ddots &  \vdots \\
    <v_0, v_{T-1}> &  \ldots & <\delta_j, v_{T-1}> &  \ldots & <v_{T-1},v_{T-1}>
    \end{vmatrix} \\
    &= \frac{1}{\det(A)} \begin{vmatrix}
    <v_0, v_0> &  \ldots & <v_k, v_0> & \ldots  &  <v_{T-1},v_0> \\
    \vdots &  \ddots  & \vdots & \ddots &  \vdots \\
    <v_0, v_{T-1}> &  \ldots & <v_k, v_{T-1}> & \ldots &  <v_{T-1},v_{T-1}>
    \end{vmatrix} \\
    &=\left\{
    \begin{array}{ll}
        1 & \mbox{if } l=k \\
        0 & \mbox{otherwise}
    \end{array}
\right.
    \end{align*}
We deduce from this 
\begin{equation*}
    <v_k, Q \overline{\alpha_l}> = \overline{\lambda_l} <v_k, \overline{\alpha_l}>, ~~\forall ~0 \leq k \leq T-1
\end{equation*}
Finally, for any $v \in Span(v_0, \ldots, v_{T-1})^\perp$,
\begin{equation*}
    <v,\alpha_l> = \frac{1}{\det(A)} \begin{vmatrix}
    <v_0, v_0> &  \ldots & 0 & \ldots  &  <v_{T-1},v_0> \\
    \vdots &  \ddots  & \vdots & \ddots &  \vdots \\
    <v_0, v_{T-1}> &  \ldots & 0 & \ldots &  <v_{T-1},v_{T-1}>
    \end{vmatrix}
    =0
\end{equation*}
Thus we have
\begin{equation*}
    <v,Q\overline{\alpha_l}> = 0 = \overline{\lambda_l} <v_k, \overline{\alpha_l}>
\end{equation*}
because $\overline{\,^tQ}v \in Span(v_0, \ldots, v_{T-1})^\perp$. \\
Hence for each $k \in \{0, \ldots, T-1\}$, there is $\beta_k \in \C$ such that $\alpha_k = \beta_k w_k$ (where $w_k$ is defined at the beginning of the subsection). We will show that $\beta_k = \beta_0 = 1$ for any $k$. \\
Remark that $A$ can be written as $\underset{i=1}{\overset{T}{\sum}} a_{i-1} P_{\sigma_i}$ where $P_{\sigma_i}$ is the permutation matrix of $\sigma_i$ where $\sigma_i = (i ~i+1 ~\ldots ~i-2 ~i-1)$ and $a_0 > 0$ and $a_1, \ldots, a_{T-1} \in \C$. In other words, $A$ is of the following shape
\begin{equation*}
    A = \begin{pmatrix}
    a_0 & a_1 & a_2 & \ldots & a_{T-1} \\
    a_{T-1} & a_0 & a_1 & \ldots & a_{T-2} \\
    \vdots & \vdots & \ddots & \vdots \\
    a_1 & a_2 & a_3 & \ldots & a_0
    \end{pmatrix}
\end{equation*}
with $a_0 > 0$ and $a_1, \ldots, a_{T-1} \in \C^{T-1}$. Moreover, since $1 \in C_0$, $<\delta_1, v_l> = <\delta_1, v_0> = \nu_1$ for any $l \in \{0, \ldots, T-1\}$. As a result, for any $l \in \{0, \ldots, T-1\}$,
\begin{align}
\label{essai0}
    \det(A) \alpha_l(1) &= \begin{vmatrix}
    a_0 &  \ldots & \nu_1 & \ldots  &  a_{T-1} \\
    \vdots &  \ddots  & \vdots & \ddots &  \vdots \\
    a_1 &  \ldots & \nu_1 & \ldots &  a_1
    \end{vmatrix} \\
\label{essai1}
    &= \begin{vmatrix}
    \nu(1) & a_{l+1} & \ldots &  \ldots  &  a_{l-1} \\
    \vdots & \vdots & \ddots  &  \ddots &  \vdots \\
    \nu(1) & a_{l+2} & \ldots &  \ldots &  a_l
    \end{vmatrix} \\
\label{essai2}
    &= \begin{vmatrix}
    \nu(1) & a_{1} & \ldots &  \ldots  &  a_{T-1} \\
    \vdots & \vdots & \ddots  &  \ddots &  \vdots \\
    \nu(1) & a_{2} & \ldots &  \ldots &  a_0
    \end{vmatrix} \\
    &=\det(A) \alpha_0(1)
\end{align}
Indeed, from \eqref{essai0} to \eqref{essai1}, we applied a circular permutation for the columns in order to put the vector $\,^t(\nu(1), \ldots, \nu(1))$ at the first column, and the determinant stays the same after this transformation. From \eqref{essai1} to \eqref{essai2}, we did a circular permutation on the rows, which does not affect either the determinant. \\
We deduce from this equality that $\beta_k = \beta_0$ for any $k \in \{0, \ldots, T-1\}$ because $w_k(1) = w_0(1)$. Concerning the fact that $\beta_0 = 1$, remark that 
\begin{align*}
    \underset{i=1}{\overset{K}{\sum}} \nu(i) \overline{\alpha_0(i)} &= \underset{i=1}{\overset{K}{\sum}} \overline{v_0(i) \alpha_0(i)} \\ &= \frac{1}{\det(A)} \begin{vmatrix}
    <v_0,v_0> & \ldots & <v_{T-1},v_{0}> \\
    \vdots & \ddots & \vdots \\
    <v_{0}, v_{T-1}> & \ldots & <v_{T-1}, v_{T-1}> 
    \end{vmatrix} \\
    &=1
\end{align*}
And 
\begin{equation*}
    \underset{i=1}{\overset{K}{\sum}} \nu(i) \overline{\alpha_0(i)} = \beta_0 \underset{i=1}{\overset{K}{\sum}} \nu(i) \xi(i) = 1
\end{equation*}
\end{proof}
The statement of Theorem 1 is meaningful provided the coefficient of the leading term $\rho^n$ is not equal to $0$. In the following proposition we prove that this coefficient is actually not $0$.  
\begin{proposition}
\label{non-nul}
For any $n \in \N$ and any $x$ 
\begin{equation*}
    \underset{l=0}{\overset{T-1}{\sum}} e^{-\frac{2inl\pi}{T}} <w_l,\delta_x> <v_l,\1> \ne 0
\end{equation*}
\end{proposition}
\begin{proof}
Let $x \in E_0$. Then there exists $k \in \{0, \ldots, T-1\}$ such that $x \in C_k$. Thus, for any $n \in \N$,
\begin{align*}
    \underset{l=0}{\overset{T-1}{\sum}} e^{-\frac{2inl\pi}{T}}\overline{w_l(x)}  <v_l,\1> 
    &= \underset{l=0}{\overset{T-1}{\sum}} e^{-\frac{2i(n+k)l\pi}{T}}\xi(x)\left(\underset{j=0}{\overset{T-1}{\sum}} \underset{y \in C_j}{\sum} e^{\frac{2i\pi l j}{T}} \nu(y)\right)  \\
     &=\underset{j=0}{\overset{T-1}{\sum}} \underset{y \in C_j}{\sum} \xi(x) \nu(y) \left( \underset{l=0}{\overset{T-1}{\sum}} e^{-\frac{2i\pi (n+k-j)l}{T}}\right) \\
    &=T\underset{T | n+k-j}{\sum}\underset{y \in C_j }{\sum}  \xi(x) \nu(y) \\&~~~~~~+  \underbrace{\underset{T \nmid n+k-j}{\sum}\underset{y \in C_j }{\sum}  \xi(x) \nu(y) e^{\frac{i\pi(n+k-j)(T-1)}{T}}\frac{\sin(\pi(n+k-j))}{\sin(\frac{\pi(n+k-j)}{T})}}_{=0} \\
    &=T\underset{T | n+k-j}{\sum}\underset{y \in C_j}{\sum}\xi(x) \nu(y) > 0
\end{align*}    
\end{proof}

\subsection{Quasi-ergodic distribution for the classical non-moving sense in the general case}
\label{2}
Now assume that the sub-transition matrix $Q$ is not necessarily irreducible. For each $x \in \{1, \ldots, K\}$, we denote by $D_x$ the subset of $\{1, \ldots, K\}$ defined by
\begin{equation*}
    D_x := \{y \in \{1, \ldots, K\} : \exists n,m \in \N, \P_x(Y_n = y) > 0 \text{ and } \P_y(Y_m = x) > 0\}
\end{equation*}
It is well-known that $(D_x)_{x \in \{1, \ldots, K\}}$ are equivalence classes. Note that,
for each $x$, the restriction of $Y$ on $D_x$ is irreducible. Thus we can associate, for each $D_x$, a period $T_x$.  We can also associate to $D_x$ a spectral radius $\rho_x$ and some left and right-eigenvectors $(v_{x,l})_{0 \leq l \leq T_x-1}$ and $(w_{x,l})_{0 \leq l \leq T_x-1}$ constructed in the same way as in the subsection \ref{1-1}. Particularly, for every $x \in \{1, \ldots, K\}$, $\nu_x := v_{x,0}$ and $\xi_x := w_{x,0}$ are vectors whose all the components are positive and such that $<\nu_x,\1>=<\nu_x,\xi_x>=1$. We define also, for any $x$, $$\varphi_x : f \to \underset{j=1}{\overset{|D_x|}{\sum}} f(j)\nu_x(j) \xi_x(j)$$ 
where $|D_x|$ is the number of elements in $D_x$. 
Now fix $\mu \in {\cal M}_1(\{1, \ldots, K\})$. Denote by $Supp(\mu)$ the support of $\mu$. Then we can define 
    $${\cal B} = \{x \in \{1, \ldots, K \}: Supp(\mu) \cap D_x \ne \emptyset\}$$
    $$\rho_{max} = \underset{x \in {\cal B}}{\max}~ \rho_x$$
 and we define ${\cal B}_{max}$ as follows
\begin{equation*}
    {\cal B}_{max} = \{x \in {\cal B} : \rho_x = \rho_{max}\}
\end{equation*}
We set the following hypothesis
\begin{hypothesis}
\label{hyp1}
There exists $x_{max} \in \{1, \ldots, K\}$ such that
\begin{equation*}
    {\cal B}_{max} = D_{x_{max}}
\end{equation*}
Under this hypothesis, the following notation will be used
\begin{equation}
\label{nu-max}
\nu_{max} = \nu_{x_{max}}
\end{equation}
\begin{equation}
\label{xi-max}
\xi_{max} = \xi_{x_{max}}
\end{equation}
\begin{equation}
\varphi_{max} = \varphi_{x_{max}}
\end{equation}
\end{hypothesis}
In all what follows, we have to keep in mind that the definition of ${\cal B}_{max}$ implicitly depends on the initial distribution $\mu$ (more precisely on the support of $\mu$). \\
The following statement explains therefore that the quasi-ergodic distribution exists if the Hypothesis \ref{hyp1} holds.
\begin{theorem}
\label{qed-thm}
Let $\mu \in {\cal M}_1(\{1, \ldots, K\})$. Then, if the Hypothesis \ref{hyp1} holds, the following convergence holds for any measurable bounded function $f :\{1, \ldots, K\} \to \R$,
\begin{equation*}
\E_\mu\left(\frac{1}{n}\underset{k=0}{\overset{n-1}{\sum}} f(Y_k) | \tau > n\right) \underset{n \to \infty}{\longrightarrow} \varphi_{max}(f)
\end{equation*}
\end{theorem}

\begin{proof}
According to Proposition \ref{quasi-ergodicity}, giving the fact that $Y$ is irreducible on each $D_x$, we have for any $x \in \{1, \ldots, K\}$
\begin{equation}
\label{zbraaa}
    \E_x\left(\frac{1}{n} \underset{k=0}{\overset{n-1}{\sum}} f(Y_k) \1_{\tau > n} \right) = \rho_x^n \varphi_x(f) \underset{l=0}{\overset{T-1}{\sum}} e^{-\frac{2inl\pi}{T_x}}  <w_{x,l},\delta_x> <v_{x,l},\1> + o(\rho_x^n) 
\end{equation}
Thus, for any $\mu \in {\cal M}_1(E)$
\begin{align*}
    &\E_\mu\left(\frac{1}{n}\underset{k=0}{\overset{n-1}{\sum}} f(Y_k) | \tau > n\right) = \frac{\underset{j=1}{\overset{K}{\sum}} \mu(j) \E_j\left(\frac{1}{n} \underset{k=0}{\overset{n-1}{\sum}} f(Y_k) \1_{\tau > n} \right)}{\underset{j=1}{\overset{K}{\sum}} \mu(j) \P_j(\tau > n)} \\
    &=\frac{\underset{j=1}{\overset{K}{\sum}} \mu(j) \rho_j^n \varphi_j(f) \underset{l=0}{\overset{T_j-1}{\sum}} e^{-\frac{2inl\pi}{T_j}}  <w_{j,l},\delta_x> <v_{j,l},\1> + o(\rho_j^n)}{\underset{j=1}{\overset{K}{\sum}} \mu(j) \rho_j^n \underset{l=0}{\overset{T_j-1}{\sum}} e^{-\frac{2inl\pi}{T_j}}  <w_{j,l},\delta_x> <v_{j,l},\1> + o(\rho_j^n)} \\
    &=\frac{\underset{j \in {\cal B}_{max}}{\sum} \varphi_j(f)\mu(j) \underset{l=0}{\overset{T_j-1}{\sum}} e^{-\frac{2inl\pi}{T_j}}  <w_{j,l},\delta_x> <v_{j,l},\1> + o(1)}{\underset{j \in {\cal B}_{max}}{\sum} \mu(j) \underset{l=0}{\overset{T_j-1}{\sum}} e^{-\frac{2inl\pi}{T_j}}  <w_{j,l},\delta_x> <v_{j,l},\1> + o(1)}\\
    &=\frac{ \varphi_{max}(f)\underset{j \in {\cal B}_{max}}{\sum}\mu(j) \underset{l=0}{\overset{T_j-1}{\sum}} e^{-\frac{2inl\pi}{T_j}}  <w_{j,l},\delta_x> <v_{j,l},\1> + o(1)}{\underset{j \in {\cal B}_{max}}{\sum} \mu(j) \underset{l=0}{\overset{T_j-1}{\sum}} e^{-\frac{2inl\pi}{T_j}}  <w_{j,l},\delta_x> <v_{j,l},\1> + o(1)}\\
\end{align*}
where the Hypothesis \ref{hyp1} was used for the last equality, implying that $\varphi_j(f) = \varphi_{max}(f)$ for all $j \in \cB_{max}$. Note moreover that this hypothesis is useful only to make this equality right. 

Then using Proposition \ref{non-nul}, we can conclude
\begin{equation*}
    \E_\mu\left(\frac{1}{n}\underset{k=0}{\overset{n-1}{\sum}} f(Y_k) | \tau > n\right) \underset{n \to \infty}{\longrightarrow} \varphi_{max}(f)
\end{equation*}
\end{proof}

\subsection{Quasi-ergodic distribution with periodic moving killing subset}
In this subsection we are interested in the quasi-ergodicity of the chain $X$ defined in the Introduction considering that the boundaries are moving $\gamma$-periodically.
We denote by $Y = (Y_n)_{n \in \N}$ the Markov chain defined on $E \times \Z/\gamma\Z$ by 
\begin{equation}
\label{y}
Y_n = (X_n,\overline{n})
\end{equation}
where $\bar n$ is the residue of $n$, modulo $\gamma$. 
$Y$ is therefore a Markov chain defined on a finite space state, which is irreducible  if and only if $gcd(T(X),\gamma) = 1$, where $T(X)$ is the period of $(X_n)_{n \in \N}$. If the chain $Y$ is actually irreducible, the associated period is  \begin{equation*}
    T = LCM(T(X),\gamma)
\end{equation*}
where $LCM(\cdot, \cdot)$ refers to the least common multiple.\\ Moreover we have
\begin{equation*}
\tau = \inf\{n \geq 0 : X_n \in A_n\} = \inf\{n \geq 0 : Y_n \in \partial\}
\end{equation*}
with $$\partial := \{(x,\overline{k}): x \in A_k\}$$
 Remark that $\partial$ is a non moving killing subset for the chain $Y$. Thus we can apply Theorem \ref{qed-thm} to the process $Y$ which yields the following theorem
\begin{theorem}
\label{qed-thm-mov}
Let $\mu \in {\cal M}_1(E_0)$. Assume that $(A_n)_{n \in \N}$ is periodic and  $Y$ defined in \eqref{y} satisfies the Hypothesis \ref{hyp1}. Then, for any measurable bounded function $f$,
\begin{equation*}
    \E_\mu\left(\frac{1}{n}\underset{k=0}{\overset{n-1}{\sum}} f(X_k) | \tau > n\right) \underset{n \to \infty}{\longrightarrow} \underset{(x,y) \in E \times \Z/\gamma\Z -\d}{\sum} f(x) \nu_{max}(x,y) \xi_{max}(x,y)
\end{equation*}
where $\nu_{max}$ and $\xi_{max}$ are the probability measures defined in \eqref{nu-max} and \eqref{xi-max} relatively to $Y$.
\end{theorem}
We can also give the following corollary which requires assumptions on $X$ and $(A_n)_{n \in \N}$.
\begin{corollary}
\label{cor-1}
Assume that $(A_n)_{n \in \N}$ is $\gamma$-periodic and that $gcd(T,\gamma)=1$ (where $T$ is the period of $X$). Then there exists $\eta \in {\cal M}_1(E)$ such that, for any $\mu \in {\cal M}_1(E_0)$ and any $f$ bounded measurable,
   $$\E_\mu\left(\frac{1}{n}\underset{k=0}{\overset{n-1}{\sum}} f(X_k) | \tau > n\right) \underset{n \to \infty}{\longrightarrow}\int f d\eta$$
\end{corollary}
\begin{proof}[Proof of Theorem \ref{qed-thm-mov}]
It is enough to apply Theorem \ref{qed-thm} to the chain $Y$ defined on \eqref{y} and to deduce the results on $X$ thanks to the following equality
\begin{equation*}
\mathbb{E}_{\mu}\left(\frac{1}{n} \underset{k=0}{\overset{n-1}{\sum}} f(X_k) | \tau > n\right) = \mathbb{E}_{\mu \otimes \delta_{\overline{0}}}\left(\frac{1}{n} \underset{k=0}{\overset{n-1}{\sum}} \tilde{f}(Y_k) | \tau > n\right), ~~ \forall n \geq 1
\end{equation*}
where $\tilde{f}$ is the projection on the first component.
\end{proof}

\section{Existence of Q-process with boundaries moving periodically}
In this section, we are interested in the Q-process, which can be interpreted as the law of the process $X$ conditioned never to be killed by the moving boundary. As before, we still consider that the boundary moves periodically period $\gamma$. \\
To show the existence of the Q-process, we will consider again the Markov chain $Y$ defined in \eqref{y}, that is defined by
\begin{equation*}
Y_n = (X_n,\overline{n}),~~\forall n \in \N
\end{equation*}
and we take back the notation introduced in subsection \ref{2} associated to $Y$. \\
The following statement ensures the existence of a Q-process even when the boundary moves. However, it is interesting to observe that we lose the homogeneity of the Q-process because of the movement of the killing boundary.
\begin{theorem}
\label{Q-process}
For any $n \in \N$ and any $x \in E_0$, the probability measure $\mathbb{Q}_x$ defined by
\begin{equation*}
\mathbb{Q}_x(X_1 \in \cdot, \ldots, X_n \in \cdot) = \underset{m \to \infty}{\lim} \P_x(X_1 \in \cdot, \ldots, X_n \in \cdot | \tau > m)
\end{equation*}
is well-defined and, under the probability $\mathbb{Q}_x$, $(X_n)_{n \in \N}$ is a time-inhomogeneous Markov chain such that for any $n \in \N$, for any $(y,z) \in E_{n-1} \times E_n$
\begin{equation*}
\mathbb{Q}_x(X_n = z | X_{n-1} = y) = \frac{\xi_x(z,\overline{n})}{\rho_x \xi_x(y,\overline{n-1})} \P_{y}(X_1 = z, \tau_{A_n} > 1)
\end{equation*}

\end{theorem}

\begin{proof}

For any $m,n \in \N$, for any $f_1, \ldots, f_n$ measurable bounded functions and for any $x \in E_0$,
\begin{align}
\E_x(f_1(Y_1) \ldots f_n(Y_n) | \tau > n+m) &= \frac{\E_x(f_1(Y_1) \ldots f_n(Y_n) \1_{\tau > n+m})}{\P_x(\tau > n+m)} \\
\label{brah}
&=\E_x\left(f_1(Y_1) \ldots f_n(Y_n) \frac{\1_{\tau > n} \P_{Y_n}(\tau > m)}{\P_x(\tau > n+m)}\right) 
\end{align}
According to the equality \eqref{zbraaa} applied to the function equal to $1$, for any $y \in E \times Z/\gamma\Z - \d$ and $n \in \N$,
$$\P_y(\tau > n) = \rho_y^n \underset{l=0}{\overset{T_y-1}{\sum}} e^{-\frac{2inl\pi}{T_y}} <w_{y,l}, \delta_y> <v_{y,l},\1> + o(\rho_y^n)$$ 
Thus, using this in \eqref{brah},
\begin{align}
&\E_x(f_1(Y_1) \ldots f_n(Y_n)| \tau > m+n) \\
\label{from}
&= \E_x\left(f_1(Y_1)\ldots f_n(Y_n) \frac{\1_{\tau > n}\rho_{Y_n}^m \underset{l=0}{\overset{T_{Y_n}-1}{\sum}} e^{-\frac{2iml\pi}{T_{Y_n}}} <w_{Y_n,l},\delta_{Y_n}> <v_{Y_n,l},\1> + o(\rho_{Y_n}^m)}{\rho_x^{n+m} \underset{l=0}{\overset{T_x-1}{\sum}} e^{-\frac{2i(n+m)l\pi}{T_x}} <w_{x,l},\delta_x> <v_{x,l},\1> + o(\rho_x^{n+m})}\right) \\
\label{to}
&= \E_x\left(f_1(Y_1)\ldots f_n(Y_n) \frac{\1_{\tau > n}\rho_{x}^m \underset{l=0}{\overset{T_{x}-1}{\sum}} e^{-\frac{2iml\pi}{T_{x}}} <w_{x,l},\delta_{Y_n}> <v_{x,l},\1> + o(\rho_{x}^m)}{\rho_x^{n+m} \underset{l=0}{\overset{T_x-1}{\sum}} e^{-\frac{2i(n+m)l\pi}{T_x}} <w_{x,l},\delta_x> <v_{x,l},\1> + o(\rho_x^{n+m})}\right) \\
&= \E_x\left(f_1(Y_1)\ldots f_n(Y_n) \frac{\1_{\tau > n} \underset{l=0}{\overset{T_{x}-1}{\sum}} e^{-\frac{2iml\pi}{T_{x}}} <w_{x,l},\delta_{Y_n}> <v_{x,l},\1> + o(1)}{\rho_x^{n} \underset{l=0}{\overset{T_x-1}{\sum}} e^{-\frac{2i(n+m)l\pi}{T_x}} <w_{x,l},\delta_x> <v_{x,l},\1> + o(\rho_x^n)}\right) 
\label{according}
\end{align}
The passage from \eqref{from} to \eqref{to} is due to the fact that, for any $n \in \N$, $Y_n \in D_x$ almost surely and the quantities $T_x$, $\rho_x$, $w_{x,l}$ and $v_{x,l}$ depends only on $D_x$. \\
Since the restriction of the chain $Y$ on $D_x$ is irreducible, we can construct as in the subsection \ref{1-1} some clusters $(C_j)_{0 \leq j \leq T_x-1}$ such that $x \in C_0$ and
\begin{equation*}
\P_x(Y_{k+nT_x} \in C_k, \tau > k+nT_x) = 1, ~~\forall k \in \{0, \ldots, T_x-1\}, \forall n \in \N
\end{equation*} 
For any $y \in D_x$, denote by $j(y)$ the integer such that $y \in C_{j(y)}$. Then we deduce from the equality \eqref{eigenvectors} in the subsection \ref{1-1} that for any $y \in E \times \Z/\gamma\Z - \d$ and $n \in \N$, 
\begin{equation*}
 e^{-\frac{2inl\pi}{T_x}} <w_{x,l},\delta_y> = e^{-\frac{2i\pi(n+j(y))l}{T_x}}\xi_x(y)
\end{equation*}
Thus, according to \eqref{according} and the previous equality,
\begin{align*}
&\E_x(f_1(Y_1) \ldots f_n(Y_n)| \tau > m+n) \\
&= \E_x\left(f_1(Y_1)\ldots f_n(Y_n) \frac{\1_{\tau > n} \xi_{x}(Y_n) \left(\underset{l=0}{\overset{T_x-1}{\sum}} e^{-\frac{2i\pi(m+j(Y_n))l}{T_x}}<v_{x,l},\1> + o(1)\right)}{{\rho}_{x}^{n} \xi_{x}(x) \left(\underset{l=0}{\overset{T_x-1}{\sum}} e^{-\frac{2i\pi(m+n+j(x))l}{T_x}}<v_{x,l},\1> + o(1)\right)}\right)
\end{align*}
However, for any $n \in \N$, 
\begin{equation*}
j(Y_n) = j(x) + n \mod T_x, ~~a.s.
\end{equation*}
and for any $m,n \in \N$,
$$\underset{l=0}{\overset{T_x-1}{\sum}} e^{-\frac{2i\pi(m+n+j(x))l}{T_x}}<v_{x,l},\1> \ne 0$$Since the state space $E \times \Z / \gamma \Z$ is finite, we may first consider function $f_i(y) = \1_{y = x_i}$, so that quantities in the ratio except $\1_{\tau > n}$ are fixed. This justifies that we can exchange the expectation and the limit as $n \to \infty$ in the previous expression.
We deduce that,
$$\E_x(f_1(Y_1) \ldots f_n(Y_n)| \tau > m+n) \underset{m \to \infty}{\longrightarrow}  \E_x\left(f_1(Y_1)\ldots f_n(Y_n) \frac{\1_{\tau > n} \xi_{x}(Y_n)}{{\rho}_{x}^{n} \xi_{x}(x,\overline{0})}\right) $$
The statement on $X$ is obtained using projection functions and we can deduce from it the transition kernel of the Q-process.

\end{proof}
\section{Example : discrete-time random walk}
\label{random-walk}
We shall illustrate the  previous results by looking at a discrete-time random walk.
Let $p \in ]0,1[$. We denote by $(M_n^{p})_{n \in \mathbb{N}}$ the Markov chain defined on $\Z$ such that
\begin{align*}
    &\P(M_{n+1}^p = M_n^p + 1 | M_n^p) = 1-p \\
    &\P(M_{n+1}^p = M_n^p - 1 | M_n^p) = p
\end{align*}
Before dealing with the quasi-ergodicity with moving boundaries, let us recall some properties about quasi-stationarity concerning random walks. For any $K \geq 1$ we define
\begin{equation*}
    T_K = \inf\{n \geq 0 : M_n^p \in (-\infty,0] \cup [K+1, \infty) \}
\end{equation*}
The sub-Markovian transition matrix associate to $(M_{n \land T_K}^{p})_{n \in \N}$ is the matrix $Q_{K} \in M_{K}(\mathbb{R})$ defined by :  
\begin{equation*}
Q_{K} = \begin{pmatrix}
0 & 1-p & 0 & \ldots & 0 & 0 \\
p & 0 & 1-p & \ldots & 0 & 0 \\
0 & p & 0 & \ldots & 0 & 0\\
\vdots & \vdots & \vdots & \ddots & \vdots & \vdots \\
0 & 0 & 0 & \ldots & 0 & 1-p \\
0 & 0 & 0 & \ldots & p & 0  
\end{pmatrix}
\end{equation*}
For any $K \geq 1$, denote by $P_K(X)$ the characteristic polynomial of $Q_{K}$. Using standard algebraic manipulations, one can show that for any $K \geq 1$, the following recurrence relation is satisfied 
\begin{equation*}
P_{K+2}(X) = -X P_{K+1}(X) - p(1-p) P_{K}(X)
\end{equation*}
with $P_1(X) = -X$ and $P_2(X) = X^2 - p(1-p)$. We set $P_0(X) = 1$.\\
For any $K \geq 0$, define
\begin{equation*}
U_K(X) = \left(-\frac{1}{\sqrt{p(1-p)}}\right)^K P_K\left(2\sqrt{p(1-p)}X\right)
\end{equation*}
Then the following equation is satisfied
\begin{equation*}
U_{K+2}(X) = 2X U_{K+1}(X) - U_K(X)
\end{equation*}
for which $U_0(X) = 1$ and $U_1(X) = 2X$. In other words, the sequence $(U_K)_{K \geq 0}$ are the Chebyshev's polynomials of the second kind and we have for any $\theta \in \mathbb{R}$
\begin{equation*}
U_K(\cos(\theta)) = \frac{\sin((K+1)\theta)}{\sin(\theta)}
\end{equation*}
The set of roots of $U_K$, hence of $P_K$, is thus well-known. It follows  
\begin{equation*}
Sp(Q_{K}) = \left\{ \lambda_j := 2\sqrt{p(1-p)}\cos\left(\frac{j \pi}{K+1}\right) : j \in \{1, \ldots, K\}\right\}
\end{equation*}
We are interested now in the eigenvectors of $Q_K$. 
\begin{proposition}
Let $K \geq 1$. Then, for any $j \in \{1, \ldots, K\}$, $\text{Ker}(Q_K - \lambda_j I_k) = Span(x_j)$ where
\begin{equation*}
x_j(i) = \left(-\frac{1}{1-p}\right)^{i-1} P_{i-1}(\lambda_j) = \left(\sqrt{\frac{p}{1-p}}\right)^{i-1} \frac{\sin\left(\frac{ij\pi}{K+1}\right)}{\sin\left(\frac{j\pi}{K+1}\right)}, ~~\forall i \in \{1, \ldots , K\}
\end{equation*}
\end{proposition}
\begin{proof}
Let $\lambda \in Sp(Q_{K})$. We want to find all the eigenvectors $x = (x(i))_{1 \leq i \leq K}$ associated to $\lambda$ such that $x(1)=1$. We will prove the proposition by double induction. \\\\
\textbf{Base case:} According to the relation $Q_{K} x = \lambda x$, we have 
\begin{equation}
\lambda x(1) = (1-p) x(2)
\end{equation}
Having $x(1)=1$, we will have therefore $x(2) = \frac{\lambda}{1-p} = -\frac{1}{1-p} P_1(\lambda)$, which conclude the base case\\\\
\textbf{Inductive step:} Let $i \in \{3, \ldots, K-1\}$. We assume that the equality is satisfied for $i-1$ and $i-2$, so we have 
\begin{align*}
&x({i-2}) = \left(-\frac{1}{1-p}\right)^{i-3} P_{i-3}(\lambda) \\
&x({i-1}) = \left(-\frac{1}{1-p}\right)^{i-2} P_{i-2}(\lambda)
\end{align*}
Using $\lambda x = Q_{K} x$, 
\begin{equation*}
\lambda x({i-1}) = px({i-2}) + (1-p)x(i)
\end{equation*}
So
\begin{align*}
x(i) &= \frac{1}{1-p}\left( \lambda x({i-1}) - px({i-2})\right) \\
&=\frac{1}{1-p}\left( \lambda \left(-\frac{1}{1-p}\right)^{i-2} P_{i-2}(\lambda) - px({i-2})\right) \\
&= \left(-\frac{1}{1-p}\right)^{i-1} \left(-\lambda P_{i-2}(\lambda) - p(1-p) P_{i-3}(\lambda)\right) \\
&=\left(-\frac{1}{1-p}\right)^{i-1} P_{i-1}(\lambda)
\end{align*}
which concludes the proof.
\end{proof}
The previous proposition gives us left and right eigenvectors of $Q_K$ : if we denote by $(v_i)_{1 \leq i \leq K}$ (respectively $(w_i)_{1 \leq i \leq K}$) the left (respectively right) eigenvectors satisfying $v_iQ_K = \lambda_i v_i$ (respectively $Q_K w_i = \lambda_i w_i$), then
\begin{align*}
    &v_i(j) = \left(\sqrt{\frac{1-p}{p}}\right)^{j-1} \frac{\sin\left(\frac{ij\pi}{K+1}\right)}{\sin\left(\frac{i\pi}{K+1}\right)} \\
    &w_i(j) = \left(\sqrt{\frac{p}{1-p}}\right)^{j-1} \frac{\sin\left(\frac{ij\pi}{K+1}\right)}{\sin\left(\frac{i\pi}{K+1}\right)}
\end{align*}
In particular, considering the spectral radius $\lambda_1$, the quasi-stationary distribution $\nu$ and the right-eigenvector $\xi$ associated to $\lambda_1$ satisfying $<\nu,\xi>=1$ are as follows:
\begin{align*}
    &\nu(j) = \frac{\left(\sqrt{\frac{1-p}{p}}\right)^{j-1} \sin\left(\frac{j\pi}{K+1}\right)}{\underset{k=1}{\overset{K}{\sum}}\left(\sqrt{\frac{1-p}{p}}\right)^{k-1} \sin\left(\frac{k\pi}{K+1}\right)} \\
    &\xi(j) = \frac{\underset{k=1}{\overset{K}{\sum}}\left(\sqrt{\frac{1-p}{p}}\right)^{k-1} \sin\left(\frac{k\pi}{K+1}\right)}{\underset{k=1}{\overset{K}{\sum}}\sin^2\left(\frac{k\pi}{K+1}\right)}\left(\sqrt{\frac{p}{1-p}}\right)^{j-1} \sin\left(\frac{j\pi}{K+1}\right)
\end{align*}
We are interested now in moving boundaries. Let $N \geq 1$  and consider the simplest case where $(A_n)_{n \in \N}$ is defined by
\begin{equation}
    A_n = \left\{
  \begin{array}{ll}
    (-\infty,0] \cup [2N,\infty) &\textit{if $n$ is even}   \\
    (-\infty,1] \cup [2N-1,\infty) &\textit{if $n$ is odd} 
  \end{array}
\right.
\label{simplest-case}
\end{equation}
Recall the previous notation 
\begin{equation}
Y^p_n = (M^p_{n \land \tau_0},\overline{n})
\end{equation}   
with $\overline{n} \in \Z/2\Z$. The chain is not irreducible (if $M^p_0$ is even, then for any $n$, $M^p_n$ have the same parity as $n$). It admits exactly two irreducible subsets:
\begin{enumerate}
    \item  ${\cal P} = \{(x,y) \in E : x+y \text{ is even}\}$
    \item  ${\cal I} = \{(x,y) \in E : x+y \text{ is odd}\}$
\end{enumerate}
But, as we can see in Figure \ref{figure}, the chain $Y^p$ behaves as a random walk on each irreducible subsets:
\begin{enumerate}
\item On $\cal P$, $Y^p$ has the same behavior as a random walk on $\mathbb{Z}$ starting from $[2,2N-2]$ absorbed by $\{1,2N-1\}$.
\item On $\cal I$, $Y^p$ has the same behavior as a random walk on $\mathbb{Z}$ starting from $[1,2N-1]$ absorbed by $\{0,2N\}$.
\end{enumerate}
Denote by $Y^p_{\cal P}$ (respectively $Y^p_{\cal I}$) the Markov chain such that for any $\mu \in {\cal M}_1({\cal P})$ (respectively ${\cal M}_1({\cal I})$)
\begin{equation*}
    \P_\mu(Y^{p}_1 \in \cdot) = \P_\mu((Y^p_{\cal P})_1 \in \cdot) \text{ (respectively }\P_\mu(Y^{p}_1 \in \cdot) = \P_\mu((Y^p_{\cal I})_1 \in \cdot) \text{)}
\end{equation*}
Let $\mu \in {\cal M}_1(E \times \Z/2\Z)$. Then there are $\lambda \in [0,1]$ and $\mu_{\cal P}, \mu_{\cal I} \in {\cal M}_1({\cal P}) \times {\cal M}_1({\cal I})$ such that
\begin{equation*}
    \mu = \lambda \mu_{\cal P} + (1-\lambda)\mu_{\cal I}
\end{equation*}
Hence we see that two cases are possible
\begin{proposition}
\begin{enumerate}
\item if $\lambda = 1$, ${\cal B}_{max} = {\cal P}$. Then $\rho_{max} = 2\sqrt{p(1-p)}\cos\left(\frac{\pi}{2(N-1)}\right)$,  and
\begin{equation*}
    \E_\mu\left(\frac{1}{n}\underset{k=0}{\overset{n-1}{\sum}} f(M^p_k) | \tau > n\right) \underset{n \to \infty}{\longrightarrow} \underset{j=2}{\overset{2N-2}{\sum}} f(j) \frac{\sin^2\left(\frac{(j-1)\pi}{2(N-1)}\right)}{\underset{k=1}{\overset{2N-3}{\sum}}\sin^2\left(\frac{k\pi}{2(N-1)}\right)}
\end{equation*} \\
\item if $\lambda \ne 1$, ${\cal B}_{max} = {\cal I}$. Then $\rho_{max} = 2\sqrt{p(1-p)}\cos\left(\frac{\pi}{2N}\right)$,  and
\begin{equation*}
    \E_\mu\left(\frac{1}{n}\underset{k=0}{\overset{n-1}{\sum}} f(M^p_k) | \tau > n\right) \underset{n \to \infty}{\longrightarrow} \underset{j=1}{\overset{2N-1}{\sum}} f(j) \frac{\sin^2\left(\frac{j\pi}{2N}\right)}{\underset{k=1}{\overset{2N-1}{\sum}}\sin^2\left(\frac{k\pi}{2N}\right)}
\end{equation*}
\end{enumerate}
\end{proposition}
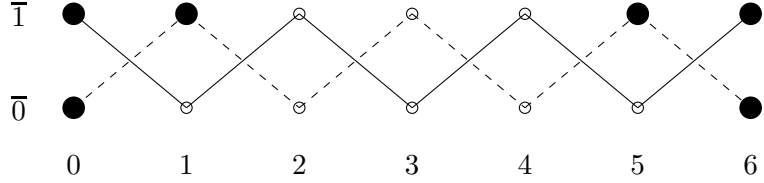
\begin{figure}
\centering
\begin{tikzpicture}
\fill (0,0) circle (0.15);
\fill (0,1.25) circle (0.15);
\draw (1.5,0) circle (0.075);
\fill (1.5,1.25) circle (0.15);
\draw (3,0) circle (0.075);
\draw (3,1.25) circle (0.075);
\draw (4.5,0) circle (0.075);
\draw (4.5,1.25) circle (0.075);
\draw (6,0) circle (0.075);
\draw (6,1.25) circle (0.075);
\draw (7.5,0) circle (0.075);
\fill (7.5,1.25) circle (0.15);
\fill (9,0) circle (0.15);
\fill (9,1.25) circle (0.15);
\draw (0,1.25) -- (1.5,0);
\draw (1.5,0) -- (3,1.25);
\draw (3,1.25) -- (4.5,0);
\draw (4.5,0) -- (6,1.25);
\draw (6,1.25) -- (7.5,0);
\draw (7.5,0) -- (9,1.25);
\draw [dashed] (0,0) -- (1.5,1.25);
\draw [dashed] (1.5,1.25) -- (3,0);
\draw [dashed] (3,0) -- (4.5,1.25);
\draw [dashed] (4.5,1.25) -- (6,0);
\draw [dashed] (6,0) -- (7.5,1.25);
\draw [dashed] (7.5,1.25) -- (9,0);
\draw (-0.5,0) node [left] {$\overline{0}$};
\draw (-0.5,1.25) node [left] {$\overline{1}$};
\draw (0,-0.5) node [below] {$0$};
\draw (1.5,-0.5) node [below] {$1$};
\draw (3,-0.5) node [below] {$2$};
\draw (4.5,-0.5) node [below] {$3$};
\draw (6,-0.5) node [below] {$4$};
\draw (7.5,-0.5) node [below] {$5$};
\draw (9,-0.5) node [below] {$6$};
\end{tikzpicture}
\caption{The black dots represent the states in $\d$. The irreducible subsets ${\cal P}$ and ${\cal I}$ are represented respectively by the dashed path and the filled path. On each path, we see that $Y^p$ behaves as a random walk.}
\label{figure} 
\end{figure}
When $(A_n)_{n \in \N}$ is moving as \eqref{simplest-case}, the quasi-ergodic distribution is the same as the non-moving quasi-ergodic distribution for one random walk absorbed at $\{0,2N\}$ except when the support of the initial distribution is included in the set of even numbers. As a matter of fact, if the chain starts from the set of even numbers, it can be absorbed only by $\{1,2N-1\}$. Remark also that the quasi-ergodic distribution of one random walk does not depend on $p$ anymore. \\\\
We have also the existence of a Q-process according to Theorem \ref{Q-process} which is the time-inhomogeneous Markov chain $(Z^p_n)_{n \in \N}$ defined by
\begin{equation*}
\P_x(Z^p_{n} = y \pm 1 | Z^p_{n-1}=y)  = 
       \frac{\sin\left(\frac{(y \pm 1)\pi}{K(y,n)}\right)}{2\sin\left(\frac{y\pi}{K(y,n)} \right)\cos\left(\frac{\pi}{K(y,n)}\right)}
\end{equation*}
with $K(y,n) = 2N-1+(-1)^{n+y}$.
\\\\

\textbf{Acknowledgement.} I would like to thank my advisor Patrick Cattiaux for his useful advices concerning the topic of this paper.

\bibliographystyle{abbrv}
\bibliography{biblio-william}

\end{document}